\documentclass[a4paper, 10pt]{amsart}

\usepackage{amsmath, amsfonts, amssymb}
\usepackage{amsthm}
\usepackage{esint}
\usepackage[
bookmarks=false]{hyperref}
\usepackage{cite}
\usepackage{array}
\usepackage{amsaddr}
\usepackage{enumerate}
\title{\large{Discreteness of the spectrum\\ of Schr\"odinger operators\\ with non-negative matrix-valued potentials}}

\author{Gian Maria Dall'Ara}
\address{Scuola Normale Superiore, Pisa, Italy}
\email{gianmaria.dallara@sns.it}

\date{\today}

\newcommand{\R}{\mathbb{R}}
\newcommand{\N}{\mathbb{N}}

\newcommand{\Z}{\mathbb{Z}}

\newcommand{\C}{\mathbb{C}}

\newcommand{\be}{\begin{equation*}}
\newcommand{\ee}{\end{equation*}}
\newcommand{\bel}{\begin{equation}}
\newcommand{\eel}{\end{equation}}
\newcommand{\bee}{\begin{eqnarray*}}
\newcommand{\eee}{\end{eqnarray*}}
\newcommand{\eps}{\varepsilon}

\newtheorem{thm}{Theorem}
\newtheorem{lem}[thm]{Lemma}
\newtheorem{prop}[thm]{Proposition}
\newtheorem{dfn}[thm]{Definition}

\begin{document}

\maketitle

\begin{abstract}
We prove three results giving sufficient and/or necessary conditions for discreteness of the spectrum of Schr\"o\-dinger operators with non-negative matrix-valued potentials, i.e., operators acting on $\psi\in L^2(\R^n,\C^d)$ by the formula\be
H_V\psi:=-\Delta\psi+V\psi,
\ee where the potential $V$ takes values in the set of non-negative Hermitian $d\times d$ matrices.

The first theorem provides a characterization of discreteness of the spectrum when the potential $V$ is in a matrix-valued $A_\infty$ class, thus extending a known result in the scalar case ($d=1$). We also discuss a subtlety in the definition of the appropriate matrix-valued $A_\infty$ class.

The second result is a sufficient condition for discreteness of the spectrum, which allows certain degenerate potentials, i.e., such that $\det(V)\equiv0$. To formulate the condition, we introduce a notion of oscillation for subspace-valued mappings.

Our third and last result shows that if $V$ is a $2\times2$ real polynomial potential, then $-\Delta+V$ has discrete spectrum if and only if the scalar operator $-\Delta+\lambda$ has discrete spectrum, where $\lambda(x)$ is the minimal eigenvalue of $V(x)$.

\end{abstract}

\section{Introduction}

\subsection{Setting and main results} The object of this paper is the analysis of Schr\"o\-dinger operators with non-negative matrix-valued potentials. Such operators act on $L^2(\R^n,\C^d)$, the space of $\C^d$-valued square-integrable functions, by the formula\bel\label{schrod}
H_V\psi(x):=-\Delta\psi(x)+V(x)\psi(x) \qquad \forall\psi\in L^2(\R^n,\C^d),
\eel where $V(x)$ is a $d\times d$ non-negative Hermitian matrix, i.e., $(V(x)v,v)\geq0$ for every $v\in\C^d$ and $x\in\R^n$. Here and in the sequel we denote by $(\cdot,\cdot)$ the canonical scalar product of $\C^d$, and by $|\cdot|$ the associated norm. The space of $d\times d$ Hermitian matrices is denoted by $H_d$, and $H_d^{\geq0}$ represents the closed cone consisting of non-negative elements of $H_d$.
The Laplacian $\Delta$ always acts componentwise on vector-valued functions: if $\psi=(\psi_1,\dots,\psi_d)$, then $\Delta\psi=(\Delta\psi_1,\dots,\Delta\psi_d)$. 

We assume that $V$ is locally integrable, in which case, by standard functional analysis, $H_V$ defines an unbounded non-negative self-adjoint operator on $L^2(\R^n,\C^d)$ (see Section \ref{f-a} for the details). 

Our aim is to study the problem of determining whether the spectrum of $H_V$ is a discrete subset of $[0,+\infty)$ consisting of eigenvalues of finite multiplicity. This is customarily expressed by saying that the spectrum of $H_V$ is discrete. 

Our interest in this question comes from the analysis of the weighted $\overline{\partial}$-problem in $\C^n$ and of the $\overline\partial_b$-problem on polynomial models of CR manifolds. The canonical solutions of these problems turn out to be related to certain Schr\"odinger operators. In particular compactness of the canonical solution operator is linked to discreteness of the spectrum of Schr\"odinger operators. This point of view has been successfully exploited when $n=1$ in \cite{christ}, \cite{berndtsson}, \cite{fu-straube}, \cite{christ-fu}. Understanding operators of the form \eqref{schrod} is a first step in the study of the more complex ones appearing in the weighted $\overline\partial$-problem when $n\geq2$ (see, e.g., \cite[Section 3]{haslinger-helffer}). 

The problem of discreteness of the spectrum of $H_V$ has been deeply studied in the scalar case ($d=1$): in 1953 Molchanov obtained in \cite{molcanov} a necessary and sufficient condition, and Maz'ya and Shubin significantly improved it in \cite{ma-sh}. See Theorem \ref{maz'ya-shubin} below for a slightly simplified statement. These characterizations only assume that the potentials are non-negative and locally integrable, but the conditions are often hard to check. Therefore it is interesting to look also for sufficient conditions that are not necessary, but easier to verify (as in \cite{simon}), or to restrict the attention to a smaller class of potentials and obtain a simpler characterization for them. 

For example, one may focus on potentials in $A_{\infty,loc}(\R^n)$, the local Muckenhoupt class. We recall that $V\in A_{\infty,loc}(\R^n)$ if and only if there exist $\delta, c>0$ and $\ell_0>0$ such that the inequality\bel\label{scalar-Ainfty}
\left|\left\{x\in Q\colon\ V(x)\geq \frac{\delta}{|Q|}\int_QV\right\}\right|\geq c|Q|
\eel holds for every cube $Q$ of side $\ell_Q\leq \ell_0$. If $V\in A_{\infty,loc}(\R^n)$, the scalar operator $H_V$ has discrete spectrum if and only if 
\be \lim_{Q\rightarrow \infty,\ \ell_Q=\ell_1}\int_{Q}V=+\infty,\ee 
for some (hence every) $\ell_1>0$. The limit is taken as the center of $Q$ goes to $\infty$, while the side $\ell_Q$ remains constant. This characterization follows from the improved Fefferman-Phong inequality of \cite[Lemma 2.1]{auscher-benali}, or from Theorem \ref{Ainfty-thm} below.\newline

In Section \ref{Ainfty-section} we introduce a matrix-valued analogue of $A_{\infty,loc}(\R^n)$, which we denote by $A_{\infty,loc}(\R^n,H_d)$ and consists of $H_d^{\geq0}$-valued locally integrable functions satisfying \eqref{scalar-Ainfty}, where the inequality between Hermitian matrices is to be interpreted in the sense of quadratic forms. Then we extend the previous characterization, proving the following theorem (see Theorem \ref{Ainfty-thm} for a more detailed statement).
\begin{thm}\label{Ainfty-intro} If $V\in A_{\infty,loc}(\R^n,H_d)$, then $H_V$ has discrete spectrum if and only if \be \lim_{Q\rightarrow \infty,\ \ell_Q=\ell_1}\lambda\left(\int_QV\right)=+\infty,\ee for some (hence every) $\ell_1>0$. Here $\lambda(A)$ denotes the minimal eigenvalue of the Hermitian matrix $A$. 
\end{thm}

An interesting point is that the scalar $A_{\infty,loc}$ class admits at least two equivalent definitions, but their natural matrix-valued analogues fail to be equivalent, giving two distinct classes of potentials. In Section \ref{Ainfty-section} we comment on this point, showing that the theorem above only holds for one of the two classes, and that is the reason of our choice of the definition of $A_{\infty,loc}(\R^n,H_d)$.

This is not the first time a matrix-valued analogue of a Muckenhoupt class appears in the literature. In order to study the vector analogue of the boundedness in weighted $L^2$ spaces of the Hilbert transform, Treil and Volberg introduced in \cite{treil-volberg} the matrix-valued analogue of the $A_2$ class. In Subsection \ref{treil-volberg-subsection} we show that the local version of their matrix class is contained in our $A_{\infty,loc}(\R^n,H_d)$, generalizing a well-known fact in the scalar theory.\newline

The leitmotiv of the rest of our paper is the comparison of the vector-valued operator $H_V=-\Delta+V$ with the scalar operator $H_\lambda=-\Delta+\lambda$, where $\lambda(x)$ is the minimal eigenvalue of $V(x)$. The following implication is elementary (see the comments after Proposition \ref{s-t-c-prop}):\bel\label{comparison}
H_\lambda \text{ has discrete spectrum}\quad\Longrightarrow \quad H_V \text{ has discrete spectrum.} \eel
In Section \ref{deg-section} we show that the reverse implication fails rather dramatically. This follows from Theorem \ref{deg-discrete} below, which provides a sufficient condition for discreteness of the spectrum that does not have any analogue in the scalar setting.

Before stating the result, we discuss the heuristics behind it. The energy of $\psi\in L^2(\R^n,\C^d)$ is the quadratic expression\be
\int_{\R^n}|\nabla\psi|^2+\int_{\R^n}(V\psi,\psi),
\ee where $|\nabla\psi|^2=\sum_{j=1}^n\sum_{k=1}^d|\partial_j\psi_k|^2$.
Showing the discreteness of the spectrum of $H_V$ amounts to proving that if $\psi$ has a large \emph{tail}, i.e., most of its $L^2$ norm is concentrated on the complement of a large ball, then it must have large energy (see Proposition \ref{s-t-c-prop} for the formal statement). For every point $x\in \R^n$, we have the orthogonal decomposition\be
\C^d=\mathcal{S}(x)\oplus\mathcal{L}(x),
\ee where $\mathcal{S}(x)$ is the direct sum of the eigenspaces of $V(x)$ corresponding to \emph{small} eigenvalues, while $\mathcal{L}(x)$ is the direct sum of the eigenspaces of $V(x)$ corresponding to \emph{large} eigenvalues. The problem is that the term $\int_{\R^n}(V\psi,\psi)$ may be small even for a $\psi$ with a large tail, if $\psi(x)$ is close to $\mathcal{S}(x)$ in an average sense. But if $\mathcal{S}(x)$ oscillates rapidly enough as a function of $x$, this can only happen if the term $\int_{\R^n}|\nabla\psi|^2$ is large, and hence only if the energy is large. 

To formalize this idea, in Subsection \ref{deg-subsection} we introduce and study the \emph{oscillation} of a measurable mapping $\mathcal{S}:\R^n\rightarrow\mathcal{V}(d)$, where $\mathcal{V}(d)$ is the set of non trivial subspaces of $\C^d$. Here by measurable we mean that there exists a finite collection of measurable vector fields $v_1,\dots,v_k:\R^n\rightarrow \C^d$ such that $\mathcal{S}(x)$ is the span of $\{v_1(x),\dots,v_k(x)\}$. The oscillation is defined for every cube $Q$ by the following formula:\be
\omega(Q,\mathcal{S}):=\inf\sqrt{\frac{1}{|Q|}\int_Q\left|v(x)-\frac{1}{|Q|}\int_Qv\right|^2dx},
\ee where the infimum is taken as $v:Q\rightarrow \C^d$ varies in the measurable unit sections of $\mathcal{S}$, i.e., the measurable mappings satisfying $|v(x)|=1$ and $v(x)\in\mathcal{S}(x)$ for almost every $x\in Q$. The quantity $\omega(Q,\mathcal{S})$ represents the $L^2$ distance of these sections from constant vector fields, and hence it is a measure of how much $\mathcal{S}$ oscillates on $Q$.

Since the discreteness of the spectrum has to do with the behavior of potentials at infinity, we introduce the following limit quantity\be
\omega_\infty(\ell,\mathcal{S})=\liminf_{x\in\Z^n,\ x\rightarrow\infty}\omega(\ell x+ [0,\ell]^n,\mathcal{S}) \qquad (\ell>0).
\ee The $\liminf$ is taken as the cubes go to infinity, while varying in what we call the \emph{canonical grid of step} $\ell$. We are finally in a position to state our sufficient condition for discreteness of the spectrum.
\begin{thm}\label{deg-discrete}
Let $V$ be a locally integrable non-negative matrix-valued potential. Assume that for every $x\in\R^n$ we have an orthogonal decomposition \be
\C^d=\mathcal{S}(x)\oplus\mathcal{L}(x),
\ee 
such that both components are measurable in $x$ and $V(x)$-invariant. If the following two properties hold:\begin{enumerate}
\item[\emph{(i)}] $\lim_{x\rightarrow\infty}\lambda(V(x)_{|\mathcal{L}(x)})=+\infty$,
\item[\emph{(ii)}] $\limsup_{\ell\rightarrow0+}\ell^{-1}\omega_\infty(\ell,\mathcal{S})=+\infty$,
\end{enumerate}
then $H_V$ has discrete spectrum.
\end{thm}

Notice that $\lambda(V(x)_{|\mathcal{L}(x)})$ denotes the minimal eigenvalue of the restriction of $V(x)$ to $\mathcal{L}(x)$, which is well-defined by the $V(x)$-invariance of this subspace.

The first condition expresses the fact that $\mathcal{L}(x)$ is a direct sum of eigenspaces corresponding to large eigenvalues of $V(x)$, while the second tells that $\mathcal{S}$ should oscillate at infinity and that, when looking at smaller and smaller scales, the oscillation should not decay too fast.

In Subsection \ref{example-deg-subsection} we show how to build potentials with rank identically equal to $1$ which satisfy the hypotheses of Theorem \ref{deg-discrete}, giving in particular a strong counterexample to the reverse implication of \eqref{comparison}.\newline

Our last result is complementary to Theorem \ref{deg-discrete} and shows that the implication \eqref{comparison} may be reversed for $2\times 2$ real polynomial non-negative potentials, i.e., non-negative potentials whose matrix entries are real-valued polynomials.

\begin{thm}\label{pol2}
If $V$ is a $2\times2$ real polynomial non-negative potential, then \bel\label{bicomparison}
H_V \text{ has discrete spectrum}\quad\Longleftrightarrow\quad  H_\lambda \text{ has discrete spectrum}. \eel
\end{thm}

The proof is based on a dichotomy (Lemma \ref{good-bad}) that allows, at every scale, to deal only with the \emph{good} cubes on which the two eigenvalues of $V$ are everywhere comparable, and the \emph{bad} ones on which the maximal eigenvalue is everywhere significantly larger than the smaller one. On the good cubes the operator behaves similarly to a scalar operator, and we can use the techniques of \cite{ma-sh}. The bad cubes are those on which the non-scalar nature of the operator becomes more apparent. The key to the proof is the observation that we can take advantage of the incomparability of eigenvalues to prove a gradient estimate for the eigenvectors (inequality \eqref{grad-v}). At this point we use in a crucial way the fact that $d=2$. The proof relies on several rigidity properties of $V$, which automatically hold for polynomials. It would be interesting to know if the result may be extended to $d\geq3$, and to study the class of potentials (for general $d$) for which \eqref{bicomparison} holds. \newline

\par We conclude the paper with Section \ref{final}, in which we explain why a natural extension of Molchanov and Maz'ya-Shubin characterization cannot hold for non-negative matrix-valued potentials, a further evidence of the new phenomena that appear when $d\geq2$. 

\subsection{Notation}\label{notation} 
If $W\in L^1(E,H_d)$ ($E\subseteq\R^n$), i.e., the space of integrable $H_d$-valued functions, then the integrability amounts, by polarization, to the integrability of $(Wv,v)$ for every $v\in\C^d$. We have $\left(\left(\int_EW\right)v,v\right)=\int_E\left(Wv,v\right)$ for every $v\in\C^d$. 

If $A,B\in H_d$, $A\leq B$ stands, as usual, for the inequality between the corresponding quadratic forms, i.e., $B-A\in H_d^{\geq0}$. 

If $W,W'\in L^1(E,H_d)$ and $W\leq W'$ on $E$, then $\int_EW\leq\int_EW'$ as quadratic forms. 

If $Q$ is a cube, by which we always mean a closed cube, then $\lambda Q$ is the cube with the same center and side multiplied by the factor $\lambda>0$. The side of $Q$ is denoted by $\ell_Q$ and $Q(x,\ell)$ is the cube of center $x$ and sides of length $\ell$ parallel to the axes. We denote by $\mathcal{G}_\ell$ ($\ell>0$) the \emph{canonical grid of step} $\ell$, i.e., the family of cubes of the form $\ell x+[0,\ell]^n$, where $x\in \Z^n$. The elements of $\cup_{N\in\Z}\mathcal{G}_{2^N}$ are the usual dyadic cubes.

ln this paper $C_n$ and $c_n$ are positive constants depending only on the dimension $n$, whose exact value may change from line to line. 

\section{Definition of $H_V$}\label{f-a}

In this section we assume that $V\in L^1_{loc}(\R^n,H_d^{\geq0})$, i.e., $V\in L^1_{loc}(\R^n,H_d)$ and $V\geq0$ everywhere. We show that the Schr\"odinger operator $H_V$ is self-adjoint on an appropriate domain by the classical quadratic form method. Our aim is to sketch the adjustments required for matrix potentials, and to state the basic characterization of discreteness of the spectrum that will be used throughout the paper. We start with the energy space
\be
E_V:=\{\psi\in L^2(\R^n,\C^d):\ \partial_j\psi\in L^2(\R^n,\C^d) \ \forall j=1,\dots,n,\quad (V\psi,\psi)\in L^1(\R^n)\},
\ee where $\partial_j$ is the distributional partial derivative. If $\psi\in E_V$, the energy\be\mathcal{E}_V(\psi):=\int_{\R^n}|\nabla\psi|^2+\int_{\R^n}(V\psi,\psi)\ee is well-defined. The expression $||\psi||_V:=\sqrt{\mathcal{E}_V(\psi)+||\psi||^2}$, where $||\cdot||$ is the $L^2(\R^n,\C^d)$ norm, is a Hilbert space norm on $E_V$. 

The subspace $C^\infty_c(\R^n,\C^d)$ is dense in $(E_V,||\cdot||_V)$.  This can be seen in the following three steps.
\begin{enumerate}
\item[(a)] Compactly supported elements of $E_V$ are dense: if $\psi\in E_V$ and $\eta\in C^\infty_c(\R^n)$ is such that $\eta(0)=1$, then $\eta(\eps x)\psi(x)$ converges to $\psi$ in $E_V$ when $\eps\rightarrow0$.
\item[(b)] Bounded elements of $E_V$ are dense: if $\psi\in E_V$, then defining $\psi_R(x):=\psi(x)$ when $|\psi(x)|\leq R$, and $\psi_R(x):=R\frac{\psi(x)}{|\psi(x)|}$ otherwise, one has $\psi_R\rightarrow \psi$ in $E_V$, by the chain rule for Sobolev spaces and dominated convergence.
\item[(c)] Test functions are dense: if $\psi\in E_V$ is bounded and compactly supported and $\{\varphi_\eps\}_{\eps>0}$ is a scalar approximate identity, then $\varphi_\eps*\psi\rightarrow\psi$ in $E_V$.\end{enumerate}
The details are routine arguments.

We define $H_V:=-\Delta+V$, where $\Delta$ acts componentwise and distributionally, on the domain\be
D(H_V):=\{\psi\in E_V:\ -\Delta\psi+V\psi\in L^2(\R^n,\C^d)\}.
\ee Thanks to the facts established above, the classical arguments (cf. \cite{fukushima}) show that $H_V$ is self-adjoint and non-negative. 
Moreover, we have a classical characterization of discreteness of the spectrum, whose proof is identical to the one in the scalar case, as can be found, e.g., in \cite{ko-sh}, p. 190-191.

\begin{prop}\label{s-t-c-prop}
$H_V$ has discrete spectrum if and only if for every $\eps>0$ there is $R<+\infty$ such that 
\bel\label{s-t-c}
\int_{|x|\geq R}|\psi|^2\leq \eps\cdot \mathcal{E}_V(\psi)\qquad\forall \psi\in C^\infty_c(\R^n,\C^d).
\eel
If \eqref{s-t-c} holds, it automatically extends to every $\psi\in E_V$. 
\end{prop} 

Proposition \ref{s-t-c-prop} immediately gives the implication \eqref{comparison} of the Introduction. First of all, notice that if $\lambda(x)$ is the minimal eigenvalue of $V(x)$ and $V\in L^1(\R^n,H_d^{\geq0})$, then $\lambda\in L^1(\R^n)$, because $\lambda(x)=\inf (V(x)u,u)$, where the $\inf$ can be taken over all vectors $u$ in a countable dense subset of the unit sphere of $\C^d$. If $\psi\in E_V$, then $|\psi|\in E_\lambda$ and $\mathcal{E}_\lambda(|\psi|)\leq \mathcal{E}_V(\psi)$, by well-known properties of Sobolev spaces (see, e.g., \cite[Ch. 4 Thm.4]{evans-gariepy}). This completes the proof of the implication \eqref{comparison}.

\section{$A_\infty$-conditions for matrix-valued functions\\ and discreteness of the spectrum}\label{Ainfty-section}

\subsection{A matrix-valued analogue of $A_{\infty,loc}(\R^n)$}

\begin{dfn}\label{Ainfloc} A function $W\in L^1_{\text{loc}}(\R^n,H_d^{\geq0})$ is in the class $A_{\infty,loc}(\R^n,H_d)$ if it satisfies the following property: there exist $\ell_0,\delta,c>0$ such that
\bel\label{Ainfest}
\left|\left\{x\in Q: W(x)\geq \frac{\delta}{|Q|}\int_QW\right\}\right|\geq c|Q|
\eel 
holds for every cube $Q$ such that $\ell_Q\leq \ell_0$.
\end{dfn}

In the scalar theory (see, e.g., \cite[Ch. 9]{grafakos}), one defines $A_\infty(\R^n)$, requiring \eqref{Ainfest} to hold for every cube $Q$. Requiring it only for small cubes, one obtains the larger class $A_{\infty,loc}(\R^n)$, which is the relevant one for our problem, as we pointed out in the Introduction.  

In the scalar case, we have the equivalent characterization: $W\in A_{\infty,loc}(\R^n)$ if and only if there exist $\alpha,\beta\in(0,1)$ and $\ell_0>0$ such that for every cube $Q$ of side $\ell_Q\leq \ell_0$ the following holds:
\bel\label{condition-b}
\forall A\subseteq Q\text{ measurable }\colon\quad|A|\geq\alpha|Q|\Longrightarrow \int_A W\geq\beta \int_Q W.
\eel A proof of the analogous equivalence of characterizations for $A_\infty(\R^n)$ can be found in \cite[Thm. 9.3.3]{grafakos}, but the same argument works in the local case too. Condition \eqref{condition-b} is meaningful also in the matrix-valued case, but it turns out to be weaker than condition \eqref{Ainfest} if $d\geq2$. To see this, we need the following brief discussion.

First of all, every $W\in A_{\infty,loc}(\R^n,H_d)$ satisfies condition \eqref{condition-b} with $\alpha=1-c/2$ and $\beta=c\delta/2$, with $c$ and $\delta$ as in Definition \ref{Ainfloc}. In fact, the intersection of $A$ with the set on the left of \eqref{Ainfest} has measure $\geq \frac{c}{2}|Q|$, hence $\int_AW\geq \frac{c\delta}{2}\int_QW$. 

Secondly, consider the set $\mathcal{W}_{n,d}$ of functions $W:\R^n\rightarrow H_d^{\geq0}$ satisfying the following properties:\begin{enumerate}
\item[(a)] the entries of $W$ are polynomials,
\item[(b)] $\det(W)\equiv0$,
\item[(c)] there is no $u\in \C^d\setminus\{0\}$ for which $Wu$ is identically zero.
\end{enumerate} 
Notice that $\mathcal{W}_{n,1}=\varnothing$, but $\mathcal{W}_{n,d}\neq\varnothing$ when $d\geq2$. An example of an element of $\mathcal{W}_{1,2}$ is $W_0(x)= \begin{bmatrix}
    1 & x \\
    x & x^2 \\
  \end{bmatrix}.$
We claim that if $W\in\mathcal{W}_{n,d}$, then there are $\alpha$ and $\beta$ such that it satisfies \eqref{condition-b} for every cube $Q$, but $W\notin A_{\infty,loc}(\R^n,H_d)$.

Let us prove the claim. Since scalar non-negative polynomials are in $A_\infty(\R^n)$ with constants depending only on the degree (see \cite[Section $2$]{ricci-stein}), we can apply \eqref{condition-b} to the family $\{(Wu,u)\}_{u\in\C^d}$, thus obtaining \eqref{condition-b} for any matrix-valued polynomial $W$. Next, observe that if $W\in\mathcal{W}_{n,d}$, then $\int_QW>0$ for every cube $Q$. In fact, if this was not the case, there would be $u\in\C^d\setminus\{0\}$ such that $\int_Q(Wu,u)=0$. The non-negativity of $W$ would then force $(Wu,u)$ to vanish identically on $Q$, and hence on $\R^n$, because it is a polynomial. This contradicts property $(c)$ above. Since $\det(W)\equiv0$ and $\int_QW>0$ on every cube, $W$ cannot satisfy \eqref{Ainfest}, thus proving the claim. \newline

\par Before stating and proving the characterization of discreteness of the spectrum for $A_{\infty,loc}$ potentials, we recall a few basic properties of $A_{\infty,loc}$, which continue to hold in the matrix setting. We begin with a consequence of inequality \eqref{condition-b}. Let $W\in L^1_{loc}(\R^n,H_d^{\geq0})$ be a function satisfying \eqref{condition-b} with constants $\ell_0$, $\alpha$ and $\beta$. If $Q''\subseteq Q'$ are cubes, $\ell_{Q'}\leq\ell_0$, and $|Q''|\geq \alpha'|Q'|$ ($0<\alpha'<\alpha$), we can find a sequence of nested cubes 
\be Q_0=Q''\subset Q_1\subset\cdots\subset Q_N=Q',\ee 
such that $|Q_k|\geq \alpha |Q_{k+1}|$ for $k=1,\dots,N-1$, and $N$ depends on $\alpha'$, $\alpha$ and $n$. Applying \eqref{condition-b} to $A=Q_k$ and $Q=Q_{k+1}$ and composing the resulting chain of inequalities, we obtain \bel\label{alphasmall}\int_{Q''}W\geq \beta' \int_{Q'}W,\eel where $\beta'$ depends on $\eps$, $\alpha$, $\beta$ and $n$. In other words, if $W$ satisfies \eqref{condition-b}, one may assume $\alpha$ arbitrarily close to $0$, at the expense of reducing the value of $\beta$. By an analogous argument, one may also increase the value of $\ell_0$ for which \eqref{condition-b} and \eqref{alphasmall} hold (at the expense of modifying the other parameters).

A consequence of this is that the precise value of $\ell_0$ in Definition \ref{Ainfloc} is inessential. In fact, if $W$ satisfies the condition of Definition \ref{Ainfloc} and $\ell_1>\ell_0$ is fixed, we can partition every cube $Q$ of side less than $\ell_1$ in $N$ smaller cubes of side less than $\ell_0$, where $N$ depends on $n$, $\ell_1$ and $\ell_0$. Applying \eqref{Ainfest} to the smaller cubes and recalling \eqref{alphasmall}, we see that $W$ satisfies \eqref{Ainfest} with $\ell_0$ replaced by $\ell_1$, and $\delta$ replaced by a smaller value depending on $\ell_1$. Since property \eqref{condition-b} follows from \eqref{Ainfest}, the value of $\ell_0$ in property \eqref{condition-b} may be analogously increased.

We can finally state the main result of this section.

\begin{thm}\label{Ainfty-thm} Let $V\in L^1_{\text{loc}}(\R^n,H_d^{\geq0})$. Consider the following conditions:
\begin{enumerate}
\item[\emph{(i)}] $-\Delta+V$ has discrete spectrum,
\item[\emph{(ii)}] for every $\ell>0$ we have \be
\lim_{x\rightarrow\infty}\lambda \left(\int_{Q(x,\ell)}V\right)=+\infty,
\ee
\item[\emph{(iii)}] there exists $\ell_1>0$ such that \bel\label{singler0}
\lim_{x\rightarrow\infty}\lambda \left(\int_{Q(x,\ell_1)}V\right)=+\infty.
\eel
\end{enumerate}
Then \emph{(i)}$\Rightarrow$\emph{(ii)}$\Rightarrow$\emph{(iii)}. If moreover $V\in A_{\infty,loc}(\R^n,H_d)$, then \emph{(iii)}$\Rightarrow$\emph{(i)}.
\end{thm}

\begin{proof}

(i)$\Rightarrow$(ii): Fix $\ell>0$ and let $\eta\in C^\infty_c(\R^n,[0,1])$ be non trivial and identically $0$ outside $Q(0,\ell)$. If $x\in \R^n$ and $u\in \C^d$ has norm $1$, we put $\eta_{x,u}(y):=\eta\left(y-x\right)u$. Fix $\eps>0$. By discreteness of the spectrum and Proposition \ref{s-t-c-prop}, there is $R$ such that \eqref{s-t-c} holds. If $Q(x,\ell)\subseteq \{|y|\geq R\}$,\bee
\int_{\R^n}\eta^2=\int_{\R^n}|\eta_{x,u}|^2&\leq&\eps\left(\int_{\R^n}|\nabla\eta_{x,u}|^2+\int_{\R^n}(V\eta_{x,u},\eta_{x,u})\right)\\
&\leq &\eps\left(\int_{\R^n}|\nabla\eta|^2+\int_{Q(x,\ell)}(Vu,u)\right).
\eee If $\eps_0$ is such that $\eps_0\int_{\R^n}|\nabla\eta|^2\leq\frac{1}{2}\int_{\R^n}\eta^2$ and $\eps\leq \eps_0$ this implies \be
\int_{\R^n}\eta^2\leq 2\eps\int_{Q(x,\ell)}(Vu,u)=2\eps\left(\left(\int_{Q(x,\ell)}V\right)u,u\right).
\ee Taking the minimum as $u$ varies on the unit sphere of $\C^d$, we get \be\lambda\left(\int_{Q(x,\ell)}V\right)\geq(2\eps)^{-1}\int_{\R^n}\eta^2.\ee
By the arbitrariness of $\eps$, we get the thesis.

(ii)$\Rightarrow$(iii): obvious.

(iii)$\Rightarrow$(i) under the $A_{\infty,loc}$ condition: Fix $N_0\in \Z$ so that $2^{N_0}\geq \ell_1$. By our previous discussion we may assume that $V$ satisfies \eqref{Ainfest} with $\ell_0=2^{N_0}$. We associate to every cube $Q$ the non-negative matrix $M(Q):=\ell_Q^{2-n}\int_QV$. By inequality \eqref{alphasmall}, there is $D$ such that if $Q'$ is a dyadic cube, $Q$ is its father and $\ell_Q\leq 2^{N_0}$, then \bel\label{doublingM} M(Q)\leq D M(Q').\eel Given $N\in\N$, by assumption there exists $R$ such that if $Q\in \mathcal{G}_{2^{N_0}}$ intersects $\{|y|\geq R\}$, we have $M(Q)\geq D^N\mathbb{I}_d$ as quadratic forms. Fix such a cube $Q$. By \eqref{doublingM} we have\be
M(Q')\geq \mathbb{I}_d \qquad\forall Q'\in \mathcal{G}_{2^{N_0-N}}: Q'\subseteq Q.
\ee Let now $Q'$ be one of the cubes above. We have
\bel\label{Q'bound}
V(x)\geq \frac{\delta}{|Q'|}\int_{Q'}V=\delta 4^{N-N_0}M(Q')\geq \delta 4^{N-N_0} \mathbb{I}_d,
\eel on a set $E(Q')\subseteq Q'$ of measure $\geq c|Q'|$. If $\psi\in C^\infty_c(\R^n,\C^d)$, we integrate the trivial inequality \be
|\psi(x)|^2\leq 2|\psi(x)-\psi(y)|^2+2|\psi(y)|^2
\ee as $(x,y)$ varies in $Q'\times E(Q')$. We get\be
|E(Q')|\int_{Q'}|\psi|^2\leq 2\int_{Q'\times E(Q')}|\psi(x)-\psi(y)|^2dxdy+ 2|Q'|\int_{E(Q')}|\psi|^2.
\ee Using \eqref{Q'bound}, the lower bound on $|E(Q')|$ and Poincar\'e inequality \be
\int_{Q'\times Q'}|\psi(x)-\psi(y)|^2\leq C_n\ell_{Q'}^2|Q'|\int_{Q'}|\nabla\psi|^2,
\ee we find\bee
c\int_{Q'}|\psi|^2&\leq& C_n\ell_{Q'}^2\int_{Q'}|\nabla\psi|^2+2\delta^{-1}4^{N_0-N}\int_{E(Q')}(V\psi,\psi)\\
&\leq& C_n4^{N_0-N}\int_{Q'}|\nabla\psi|^2+2\delta^{-1}4^{N_0-N}\int_{Q'}(V\psi,\psi).
\eee In the second line we used the non-negativity of $V$.
Summing over $Q'$ and then over all $Q\in\mathcal{G}_{2^{N_0}}$ intersecting $\{|y|\geq R\}$, we obtain  \be
\int_{|y|\geq R}|\psi|^2\leq C_nc^{-1}\delta^{-1}4^{N_0-N}\mathcal{E}_V(\psi).
\ee By the arbitrariness of $N$, we have the thesis.\end{proof}

We now show that one cannot extend the previous characterization to the alternative $A_{\infty,loc}$ class obtained using condition \eqref{condition-b}, instead of Definition \ref{Ainfloc}. Recalling the discussion following Definition \ref{Ainfloc}, it is enough to exhibit, for every $n\geq1$ and $d\geq2$, a polynomial $W:\R^n\rightarrow H_d^{\geq0}$ such that $H_W$ has not discrete spectrum, but \bel\label{limit-W}
\lim_{Q\rightarrow \infty,\ \ell_Q=\ell}\lambda\left(\int_QW\right)=+\infty \qquad\forall \ell>0.
\eel
If $n=1$ and $d=2$, we consider $W_1(x)= \begin{bmatrix}
    x^4 & x^5 \\
    x^5 & x^6 \\
  \end{bmatrix}=x^4\begin{bmatrix}
    1 & x \\
    x & x^2 \\
  \end{bmatrix}$. By Theorem \ref{pol2}, which will be proved in Section \ref{pol-section}, $H_{W_1}$ has not discrete spectrum, because $\lambda=\lambda(W_1)\equiv0$, and $H_\lambda=-\Delta$ has not discrete spectrum. For every $x\in\R$ and $\ell>0$, we have\bee
 \det\left(\int_{x-\ell/2}^{x+\ell/2}W_1(y)dy\right) &=&
  \det
  \begin{bmatrix}
   \frac{(x+\ell/2)^5-(x-\ell/2)^5}{5} & \frac{(x+\ell/2)^6-(x-\ell/2)^6}{6} \\
    \frac{(x+\ell/2)^6-(x-\ell/2)^6}{6}  &  \frac{(x+\ell/2)^7-(x-\ell/2)^7}{7} \\
  \end{bmatrix}=\frac{\ell^4x^8}{12}\\&+&\text{terms of lower degree in $x$}, \\
 \text{tr}\left(\int_{x-\ell/2}^{x+\ell/2}W_1(y)dy\right) &=& \ell x^6+\text{terms of lower degree in $x$}.
  \eee
Since $\lambda(A)\geq\frac{\det(A)}{\text{tr}(A)}$ for every non-negative $2\times 2$ matrix $A$, we conclude that 
\bel\label{diverge}\liminf_{x\rightarrow\pm\infty} \frac{\lambda\left(\int_{x-\ell/2}^{x+\ell/2}W_1(y)dy\right)}{x^2}\geq\frac{\ell^3}{12},\eel which is stronger than \eqref{limit-W}. 

If $n\geq 1$ and $d=2$, we consider instead the potential $W_n(x)=\sum_{j=1}^nW_1(x_j)$. Notice that $H_{W_n}$ is the sum of the $n$ pairwise commuting operators obtained by letting $H_{W_1}$ act on each variable separately. By basic spectral theory, $H_{W_n}$ has not discrete spectrum. To see that \eqref{limit-W} is satisfied by $W_n$, we use the concavity of $\lambda$: \be
\lambda(A+B)\geq\lambda(A)+\lambda(B)\qquad\forall A,B\geq0,\ee which follows from the fact that $\lambda$ is the infimum of the family of linear functionals $\{A\mapsto(Au,u)\}_{|u|=1}$. 
For every $x\in\R^n$ and $\ell>0$, we have\bee
 \lambda\left(\int_{Q(x,\ell)}W_n(y)dy\right) &=&  \lambda\left(\ell^{n-1}\sum_{j=1}^n\int_{x_j-\ell/2}^{x_j+\ell/2}W_1(y)dy\right) \\
 &\geq & \sum_{j=1}^n\ell^{n-1}\lambda\left(\int_{x_j-\ell/2}^{x_j+\ell/2}W_1(y)dy\right).
 \eee  By \eqref{diverge}, we conclude that
\be\liminf_{x\rightarrow\infty} \frac{\lambda\left(\int_{Q(x,\ell)}W_n(y)dy\right)}{|x|^2}\geq c_n\ell^{n+2}.\ee

For the remaining case $n\geq 1$, $d\geq 3$, we put $W_{n,d}(x)=\begin{bmatrix}
  W_n(x) & \mathbb{O}_{2,d-2} \\
   \mathbb{O}_{d-2,2}  &  |x|^2\mathbb{I}_{d-2} \\
  \end{bmatrix}$, where $\mathbb{O}$ and $\mathbb{I}$ are the zero and identity matrices of the dimensions indicated by the subscripts. We omit the elementary verification that $H_{W_{n,d}}$ has not discrete spectrum, and that $W_{n,d}$ satisfies \eqref{limit-W}.
 
\subsection{Relation with the Treil-Volberg matrix-valued $A_2$ class}\label{treil-volberg-subsection}

In \cite{treil-volberg}, Treil and Volberg introduced a Muckenhoupt $A_2$ class of non-negative matrix-valued functions. In this subsection we prove that the local version of this class is contained in $A_{\infty,loc}(\R^n,H_d)$. 

\begin{dfn}[cf. \cite{treil-volberg} ]\label{treil-volberg-def}
A measurable function $W:\R^n\rightarrow H_d^{\geq0}$ is in $A_{2,loc}(\R^n,H_d)$ if it is almost everywhere invertible, both $W$ and $W^{-1}$ are locally integrable and there exists $\ell_0>0$ such that\be
[W]_{A_{2,\text{loc}}}:=\sup_Q \left|\left| \left(\frac{1}{|Q|}\int_Q W\right)^{\frac{1}{2}}\left(\frac{1}{|Q|}\int_Q W^{-1}\right)^\frac{1}{2} \right|\right|_{op}<+\infty,
\ee where $Q$ varies over all cubes of side $\ell_Q\leq \ell_0$.
\end{dfn}

\begin{prop}
$A_{2,loc}(\R^n,H_d)\subseteq A_{\infty,loc}(\R^n,H_d)$.
\end{prop}

\begin{proof}
We estimate the measure of the complement of the set appearing in \eqref{Ainfest}. If $A$ and $B$ are two non-negative matrices, $A\ngeq B$ is not equivalent to $A<B$ (except in the case $d=1$), but we can use the fact that when $A$ is invertible, $A\geq B$ if and only if $||B^{\frac{1}{2}}A^{-\frac{1}{2}}||_{op}\leq 1$ (Lemma $V.1.7$ of \cite{bhatia}). Since $W(x)>0$ almost everywhere, this allows to write (for $Q$ any cube of side $\ell_Q\leq \ell_0$)\bee
&&\left|\left\{x\in Q: W(x)\ngeq \frac{\delta}{|Q|}\int_QW\right\}\right|=\left|\left\{x\in Q: \left|\left|\left(\frac{\delta}{|Q|}\int_QW\right)^{\frac{1}{2}}W(x)^{-\frac{1}{2}}\right|\right|_{op}>1\right\}\right|\\
&\leq& \int_Q\left|\left|\left(\frac{\delta}{|Q|}\int_QW\right)^{\frac{1}{2}}W(x)^{-\frac{1}{2}}\right|\right|_{op}^2dx\\
&=& \delta\int_Q\left|\left|\left(\frac{1}{|Q|}\int_QW\right)^{\frac{1}{2}}W(x)^{-1}\left(\frac{1}{|Q|}\int_QW\right)^{\frac{1}{2}}\right|\right|_{op}dx,
\eee where the last equality follows from the identity $||A^*A||_{op}=||A||^2$ which holds for any matrix $A$. Since $\int_Q||U||_{op}\leq d||\int_QU||_{op}$ for any integrable $U:Q\rightarrow H_d^{\geq0}$ (Lemma $3.1$ of \cite{treil-volberg}), we find\bee
&&\delta\int_Q\left|\left|\left(\frac{1}{|Q|}\int_QW\right)^{\frac{1}{2}}W(x)^{-1}\left(\frac{1}{|Q|}\int_QW\right)^{\frac{1}{2}}\right|\right|_{op}dx\\
&\leq& \delta d|Q|\left|\left|\left(\frac{1}{|Q|}\int_QW\right)^{\frac{1}{2}}\frac{1}{|Q|}\int_QW(x)^{-1}dx\left(\frac{1}{|Q|}\int_QW\right)^{\frac{1}{2}}\right|\right|_{op}\\
&=&\delta d|Q|\left|\left|\left(\frac{1}{|Q|}\int_QW\right)^{\frac{1}{2}}\left(\frac{1}{|Q|}\int_QW(x)^{-1}dx\right)^{\frac{1}{2}}\right|\right|_{op}^2\leq \delta d[W]_{A_{2,loc}}^2|Q|.
\eee
Putting everything together, we conclude that \be
\left|\left\{x\in Q: W(x)\geq \frac{\delta}{|Q|}\int_QW\right\}\right|> (1-\delta d[W]_{A_{2,loc}}^2)|Q| \qquad\forall Q:\ell_Q\leq \ell_0.
\ee When $\delta$ is small enough, this is inequality \eqref{Ainfest}.
\end{proof}

It is worth noticing that matrix-valued $A_p$ classes for $p\notin\{2,\infty\}$ were introduced in \cite{volberg}. The definition is rather different from the one of $A_2$ and it may be interesting to know how they are related to $A_{\infty,loc}(\R^n,H_d)$.

\section{A sufficient condition for discreteness of the spectrum\\ and degenerate potentials}\label{deg-section} 

\subsection{A notion of oscillation for subspace-valued mappings}\label{deg-subsection} Recall that we denote by $\mathcal{V}(d)$ the set of non trivial linear subspaces of $\C^d$. Let \be\mathcal{S}:\R^n\longrightarrow \mathcal{V}(d)\ee be a fixed measurable mapping. We recall that by measurable we mean that there exist $k$ measurable vector fields $v_1,\dots,v_k:\R^n\rightarrow \C^d$ such that $\mathcal{S}(x)$ is the span of $\{v_1(x),\dots,v_k(x)\}$.
If $Q\subseteq\R^n$ is a cube, the set of unit sections of $\mathcal{S}$ on $Q$ is\be
\Gamma(Q,\mathcal{S}):=\{v:Q\rightarrow \C^d \text{ meas.}\colon v(x)\in\mathcal{S}(x)\text{ and } |v(x)|=1\quad \text{a.e. } x\in Q\}.
\ee Non-triviality and measurability in $x$ of $\mathcal{S}(x)$ guarantee that there are plenty of unit sections. We are ready to give our key definition.

\begin{dfn}\label{def-omega} The oscillation of $\mathcal{S}$ on a cube $Q$ is the following quantity:
\bel
\omega(Q,\mathcal{S}):=\inf_{v\in\Gamma(Q,\mathcal{S})}\sqrt{\frac{1}{|Q|}\int_Q\left|v(y)-\frac{1}{|Q|}\int_Qv\right|^2dy}.
\eel 
\end{dfn}

Notice that \be
\sqrt{\frac{1}{|Q|}\int_Q\left|v(y)-\frac{1}{|Q|}\int_Qv\right|^2dy}=\inf_{b\in\C^d}\sqrt{\frac{1}{|Q|}\int_Q\left|v(y)-b\right|^2dy}
\ee and hence $\omega(Q,\mathcal{S})$ is the distance in $L^2(Q,\C^d)$ between the set $\Gamma(Q,\mathcal{S})$ and the set of constant vector fields on $Q$. By using the fact that $|v|\equiv1$, it is also easy to verify that 
\bel\label{id-osc}
\omega(Q,\mathcal{S})^2= 1-\sup_{v\in\Gamma(Q,\mathcal{S})}\left|\frac{1}{|Q|}\int_Qv\right|^2.
\eel It is then obvious that $\omega(Q,\mathcal{S})\in[0,1]$. The following lemma gives us an idea of which feature of $\mathcal{S}$ on $Q$ is measured by $\omega(Q,\mathcal{S})$.

\begin{lem}\label{omega-0} $\omega(Q,\mathcal{S})=0$ if and only if there is a subset $F\subseteq Q$ such that $|Q\setminus F|=0$ and\bel\label{intersection}
\cap_{x\in F}\mathcal{S}(x)\neq \{0\},
\eel i.e., there exists $u\in\C^d\setminus \{0\}$ such that $u\in \mathcal{S}(x)$ for almost every $x\in Q$.

\end{lem}

\begin{proof}
The non-trivial direction is the \emph{only if}. If $\omega(Q,\mathcal{S})=0$, by \eqref{id-osc} there is a sequence $\{v_k\}_{k\in\N}\subseteq \Gamma(Q,\mathcal{S})$ such that $\lim_{k\rightarrow+\infty}\left|\frac{1}{|Q|}\int_Qv_k\right|=1$. Since $\left\{\frac{1}{|Q|}\int_Qv_k\right\}_{k\in\N}$ is contained in the unit ball of $\C^d$, passing to a subsequence we can suppose that there is $u_0$ in the unit sphere of $\C^d$ to which it converges as fast as we like. In particular, we can assume that \bel\label{k^2}
\frac{1}{|Q|}\int_Q\Re (v_k,u_0)\geq 1-\frac{1}{k^2}\qquad\forall k\in\N,
\eel where $\Re(v_k,u_0)$ denotes the real part of the scalar product, which is pointwise $\leq 1$ by the Cauchy-Schwarz inequality. Consider the sets \be A_{k,m}:=\{x\in Q: \Re(v_k(x),u_0)> 1-1/m\} \qquad (k,m\in\N).\ee It is easy to see that \eqref{k^2} implies $|Q\setminus A_{k,m}|\leq \frac{m}{k^2}|Q|$. Since $\sum_{k=1}^{+\infty}|Q\setminus A_{k,m}|<+\infty$ for every $m$, the Borel-Cantelli Lemma gives\be
|\cup_{\ell\in \N}\cap_{k\geq \ell}A_{k,m}|=|Q|\qquad\forall m,
\ee and hence \be
|\cap_{m\in\N}\cup_{\ell\in \N}\cap_{k\geq \ell}A_{k,m}|=|Q|.
\ee Unravelling the notation, this means that there is a set of full measure $F$ such that $\Re(v_k(x),u_0)$ converges to $1$ at every $x\in F$. Since $|u_0|=|v_k|\equiv1$, the strict convexity of the sphere implies that $v_k(x)$ converges to $u_0$ at every $x\in F$. Since we may assume that $v_k(x)\in\mathcal{S}(x)$ for every $k$ and every $x\in F$, we conclude that $u_0\in\cap_{x\in F}\mathcal{S}(x)$, as we wanted.
\end{proof}

The following lemma states that if $\cap_{x\in Q}\mathcal{S}(x)=\{0\}$ in a certain quantitative sense, then $\omega(Q,\mathcal{S})$ has an explicit lower bound.

\begin{lem}
\emph{(i)} Let $\mathcal{S}:\R^n\rightarrow\mathcal{V}(d)$ be measurable and let $Q$ be a cube. Assume that for some $N\in\N$ we have the disjoint union \be Q=\cup_{j=1}^NA_j,\ee and $\eta>0$ is such that \be|A_j|\geq\eta|Q| \qquad\forall j.\ee 
Assume moreover that there is a $\delta>0$ such that the following property holds for each $j$: for every $x\in A_j$ and $v\in \mathcal{S}(x)$, there exists $k\neq j$ such that\bel\label{angle}
|(v,w)|\leq(1-\delta)|v||w| \qquad\forall w\in\mathcal{S}(y),\ \forall y\in A_k.
\eel Then \bel\label{lower-bound-omega}
\omega(Q,\mathcal{S})\geq \sqrt{\delta\eta}.
\eel

\emph{(ii)} If $\mathcal{S}$ is constant on each $A_j$, i.e., $\mathcal{S}\equiv\mathcal{S}_j$ on $A_j$, then the property above is satisfied for some $\delta>0$ if and only if $\cap_{j=1}^N\mathcal{S}_j=\{0\}$. The value of $\delta$ may be chosen to depend only on the subspaces $\{\mathcal{S}_j\}_{j=1}^N$.
\end{lem}

\begin{proof}
(i): Let $u\in\Gamma(Q,\mathcal{S})$. Then \bel\label{angle-est}
\left|\int_Qu\right|^2=\left|\sum_{j=1}^N\int_{A_j}u\right|^2=\sum_{j=1}^N\left|\int_{A_j}u\right|^2+\sum_{j=1}^N\int_{A_j}\int_{A_j^c}(u(x),u(y))dydx.
\eel 
Fix $x\in A_j$ such that $u(x)\in\mathcal{S}(x)$ (almost every $x$ has this property). By the hypothesis, there exists $k$ such that \eqref{angle} holds. Since $|A_k|\geq\eta|Q|$, we have\bee
\int_{A_j^c}(u(x),u(y))dy&=&\int_{(A_j\cup A_k)^c}(u(x),u(y))dy+\int_{A_k}(u(x),u(y))dy\\
&\leq& (|Q|-|A_j|-|A_k|)+|A_k|(1-\delta)\\
&=&(|Q|-|A_j|)-|A_k|\delta\leq (|Q|-|A_j|)-\delta\eta|Q|,
\eee where in the second line we used the Cauchy-Schwartz inequality and \eqref{angle}. The identity \eqref{angle-est} then gives\bee
\left|\int_Qu\right|^2&\leq& \sum_{j=1}^N|A_j|^2+\sum_{j=1}^N |A_j|(|Q|-|A_j|)-\sum_{j=1}^N|A_j|\delta\eta|Q|\\
&=&\left(\sum_{j=1}^N|A_j|\right)|Q|(1-\delta\eta)=(1-\delta\eta)|Q|^2.
\eee Recalling \eqref{id-osc} and the fact that $u$ is arbitrary, we get \eqref{lower-bound-omega}.
\ \newline

(ii): If $\cap_{j=1}^N\mathcal{S}_j\neq\{0\}$, one can take a non-zero vector $v$ in the intersection to see that \eqref{angle} cannot be satisfied for any $\delta>0$. 

Now assume that the property does not hold for a given $\delta>0$. In this case there is $j_\delta$ and $v_\delta\in\mathcal{S}_{j_\delta}$ of unit norm such that for every $k\neq j_\delta$, there exists $w_{\delta,k}\in\mathcal{S}_k$ of unit norm satisfying \bel\label{contradicting-prop}
|(v_\delta,w_{\delta,k})|>(1-\delta).
\eel If the property does not hold for any $\delta$, extracting a subsequence $\delta_n$ such that $\delta_n\rightarrow0$ we can assume that $j_{\delta_n}\equiv j_0$, $v_{\delta_n}$ converges to $v$, and $w_{\delta_n,k}$ converges to $w_k$ for every $k\neq j_0$. Passing to the limit in \eqref{contradicting-prop}, we find $|(v,w_k)|=1$. Since $|v|=|w_k|=1$, the equality condition for the Cauchy-Schwarz inequality tells us that $w_k=e^{i\phi_k}v$ for some $\phi_k\in\R$. Hence $v\in \cap_{j=1}^N\mathcal{S}_j$, i.e., the intersection is not trivial.

The last observation about the value of $\delta$ follows by observing that in the above argument we never used the sets $A_j$. \end{proof}

We proceed to give the proof of Theorem \ref{deg-discrete}.

\subsection{Proof of Theorem \ref{deg-discrete}}
We simplify a bit the notation, writing $\omega(Q)$ and $\omega_{\infty}(\ell)$ in lieu of $\omega(Q,\mathcal{S})$ and $\omega_{\infty}(\ell,\mathcal{S})$, respectively. Recall that, for a fixed scale $\ell>0$, we define
\be
\omega_\infty(\ell):=\liminf_{Q\in\mathcal{G}_\ell,\ Q\rightarrow\infty}\omega(Q).
\ee  By Proposition \ref{s-t-c-prop}, our goal is to prove that for every $\eps>0$ there exists $R<+\infty$ such that\bel\label{deg-stc}
\int_{|x|\geq R}|\psi|^2\leq \eps \left(\int_{\R^n}|\nabla\psi|^2+\int_{\R^n}(V\psi,\psi)\right)\qquad\forall \psi\in C^\infty_c(\R^n;\C^d).
\eel First of all we decompose $\psi$ in its radial and angular part, writing 
\be\psi=\phi u,\ee 
where $\phi=|\psi|$ is scalar and $u$ has unit norm. Notice that $u=\frac{\psi}{|\psi|}$ on $\Omega:=\{\psi\neq0\}$ and  may be arbitrarily defined on $\Omega^c$: this ambiguity will not affect the rest of the proof. Both $\phi$ and $u$ are smooth on $\Omega$ and there we have\be
|\partial_j\psi|^2=|(\partial_j\phi) u+\phi(\partial_ju)|^2=|(\partial_j\phi)u|^2+|\phi(\partial_ju)|^2=|\partial_j\phi|^2+\phi^2|\partial_ju|^2.
\ee The second equality follows from the orthogonality of $u$ and $\partial_ju$, a consequence of the fact that $|u|\equiv1$. Since $\nabla\psi(x)=0$ at almost every $x$ such that $\psi(x)=0$ and the same holds for $\phi$,  we can write \bel\label{Veff}
\int_{\R^n}|\nabla\psi|^2+\int_{\R^n}(V\psi,\psi)=\int_{\R^n}|\nabla\phi|^2+\int_{\R^n}\widetilde{V}\phi^2,
\eel
where $\widetilde{V}:=(Vu,u)+|\nabla u|^2$ on $\Omega$ and $\widetilde{V}:=0$ on $\Omega^c$. Notice that $\phi$ is not smooth, but it is Lipschitz and compactly supported, so that the identity above is meaningful. The non-negative function $\widetilde{V}$ is the ``effective" potential to which the scalar wave-function $\phi$ is subject. This potential is not locally integrable in general, but $\int\widetilde{V}\phi^2$ is finite, as a consequence of the computations above.

Fix $\eps>0$. By assumption (ii), we can choose $\ell>0$ such that $\ell^{-1}\omega_\infty(\ell)\geq \eps^{-1}$, and by assumption (i), we can then choose $R$ so that for every cube $Q\in\mathcal{G}_\ell$ intersecting $\{|y|\geq R\}$, we have  \bel\label{VlargeonL}
(V(x)v,v)\geq \ell^{-2}|v|^2\qquad\forall v\in \mathcal{L}(x),\ \forall x\in Q.
\eel 
Enlarging $R$ if necessary, we can assume that for every such cube $Q$, we have\bel\label{omega-inf}
\omega(Q)\geq \frac{\omega_\infty(\ell)}{2}.
\eel Fix $Q\in \mathcal{G}_\ell$ such that $Q\cap\{|y|\geq R\}\neq0$. We are going to prove that:\bel\label{eff-cubes}
\ell^{-2}\omega_\infty(\ell)^2\int_Q\phi^2\leq C_n\left( \int_Q|\nabla\phi|^2+\int_Q\widetilde{V}\phi^2\right).
\eel Summing over $Q$ and recalling identity \eqref{Veff} and our choice of $\ell$, we see that this implies \eqref{deg-stc}, as we wanted. We can turn to the proof of \eqref{eff-cubes}. 

To analyze separately the contributions of the two terms on the right hand side of \eqref{eff-cubes}, we use the technique of Lemma $2.2$ of \cite{ko-ma-sh}: we consider the compact set\be
F:=\left\{x\in Q: \phi(x)^2\leq\frac{1}{4|Q|}\int_Q\phi^2\right\}.
\ee Let $c_n$ be a small constant depending only on $n$ and to be fixed later: we split our analysis depending on whether the capacity of $F$ is smaller or larger than $c_n\cdot\ell^{n-2}\omega_\infty(\ell)^2$. We recall that, given $F\subseteq Q$, if $n\geq3$,\bel\label{cap3}
\text{Cap}(F):=\inf\left\{\int_{\R^n}|\nabla \eta|^2\colon \eta\in C^\infty_c(\R^n,[0,1]) \text{ such that }\eta=1 \text{ on } F\right\},
\eel while if $n=2$,\bel\label{cap2}
\text{Cap}(F):=\inf\left\{\int_{\R^n}|\nabla \eta|^2\colon \eta\in C^\infty_c(2Q,[0,1]) \text{ such that }\eta=1 \text{ on } F\right\},
\eel which is the capacity relative to $2Q$. In the (easier) case $n=1$, we put $\text{Cap}(F):=0$ if $F=\varnothing$ and $\text{Cap}(F)=\ell_Q^{-1}$ if $F\neq\varnothing$ ($Q$ is an interval in this case).\newline

\par If $\text{Cap}(F)\geq c_n\cdot\ell^{n-2}\omega_\infty(\ell)^2$, we can apply Lemma $2.2$ of \cite{ko-ma-sh} (or rather inequality $(2.7)$ in the proof of that Lemma) to obtain
\be
c_n\cdot\ell^{-2}\omega_\infty(\ell)^2\int_Q\phi^2\leq\frac{\text{Cap}(F)}{|Q|}\int_Q\phi^2 \leq C_n\int_Q|\nabla\phi|^2, 
\ee which implies \eqref{eff-cubes}. If $n=1$, recall that if $\phi$ is Lipschitz on a closed interval $I$ of length $\ell$ and $\phi$ vanishes at a point of $I$, then $\int_I|\phi|^2\leq \ell^2\int_I|\phi'|^2$.\newline

\par If $\text{Cap}(F)\leq c_n\cdot\ell^{n-2}\omega_\infty(\ell)^2$, we define \be
\mathcal{S}_{\text{unit}}(x):=\{v\in\C^d: v\in \mathcal{S}(x)\text{ and }|v|=1\}
\ee
and we split the analysis into two further sub-cases, depending this time on whether the quantity \bel\label{intdist}
\frac{1}{|Q|}\int_{Q\setminus F}\text{dist}(u(x),\mathcal{S}_{\text{unit}}(x))^2dx
\eel is smaller or larger than $c_n\cdot\omega_\infty(\ell)^2$. The integral \eqref{intdist} measures in $L^2$ average and on $Q\setminus F$ how far $u$ is from the distribution of subspaces $\mathcal{S}$. Notice that if $n=1$, taking $c_n\leq1/2$, we have $F=\varnothing$ and part of what follows becomes more elementary.

If \eqref{intdist} $\geq c_n\cdot\omega_\infty(\ell)^2$, we write \be
u(x)=u'(x)+u''(x),\quad u'(x)\in\mathcal{S}(x),\ u''(x)\in\mathcal{L}(x),\qquad\forall x\in Q.
\ee
By the invariance under $V$ of $\mathcal{S}$ and $\mathcal{L}$ and the non-negativity of $V$, we have \be
(Vu,u)=(Vu',u')+(Vu'',u'')\geq (Vu'',u'')\geq \ell^{-2}|u''|^2\qquad\text{ on }Q.
\ee The last inequality follows from \eqref{VlargeonL}. The elementary inequality \be
|u''(x)|^2\geq\frac{1}{2}\text{dist}(u(x),\mathcal{S}_{\text{unit}}(x))^2
\ee allows to estimate\bee
&&\int_Q\widetilde{V}\phi^2\geq \int_{Q\setminus F}(Vu,u)\phi^2 \geq \frac{1}{4|Q|}\int_{Q\setminus F}(Vu,u)\cdot\int_Q\phi^2 \\
&\geq& \frac{\ell^{-2}}{8}\frac{1}{|Q|}\int_{Q\setminus F}\text{dist}(u(x),\mathcal{S}_{\text{unit}}(x))^2dx\cdot\int_Q\phi^2\geq \frac{c_n}{8}\ell^{-2}\omega_\infty(\ell)^2\int_Q\phi^2.
\eee In the first line we used the fact that $F$ contains the zero set of $\psi$, and hence $\widetilde{V}\geq (Vu,u)$ on $Q\setminus F$. Inequality \eqref{eff-cubes} is proved if \eqref{intdist} $\geq c_n\cdot\omega_\infty(\ell)^2$.

The remaining case to be analyzed is: \eqref{intdist} $\leq c_n\cdot\omega_\infty(\ell)^2$. We are going to prove that if $c_n$ is small enough, then \bel\label{grad-u}
\frac{1}{|Q|}\int_{Q\setminus F}|\nabla u|^2\geq\frac{\ell^{-2}\omega_\infty(\ell)^2}{C_n}.
\eel
Since $\int_Q\widetilde{V}\phi^2\geq \frac{1}{4|Q|}\int_{Q\setminus F}|\nabla u|^2\int_Q\phi^2$, this concludes the proof of Theorem \ref{deg-discrete}.

By the measurability of $\mathcal{S}$ and the hypothesis on \eqref{intdist}, we can find $v\in\Gamma(Q,\mathcal{S})$ such that \bel\label{def-v} 
\frac{1}{|Q|}\int_{Q\setminus F}|u-v|^2dx\leq c_n\cdot\omega_\infty(\ell)^2.
\eel Moreover, since we are under the assumption that $F$ has small capacity, there exists $\eta\in C^\infty_c(\R^n,[0,1])$ such that $\eta\equiv1$ on $F$ (and $\eta\equiv0$ on $(2Q)^c$ if $n=2$) and \bel\label{grad-eta}
\int_{\R^n}|\nabla\eta|^2\leq 2c_n\cdot\ell^{n-2}\omega_\infty(\ell)^2.
\eel By H\"older inequality and Sobolev embedding, we have\bel\label{L2-eta}
\frac{1}{|Q|}\int_Q\eta^2\leq C_nc_n\cdot\omega_\infty(\ell)^2.
\eel 
Recalling \eqref{omega-inf}, we have\bee
\frac{\omega_\infty(\ell)}{2}&\leq& \omega(Q)\leq \sqrt{\frac{1}{|Q|}\int_Q\left|v-\frac{1}{|Q|}\int_Q v\right|^2}\\
&\leq &\sqrt{\frac{1}{|Q|}\int_Q\left|(1-\eta)v-\frac{1}{|Q|}\int_Q (1-\eta)v\right|^2}+\sqrt{\frac{1}{|Q|}\int_Q\left|\eta v-\frac{1}{|Q|}\int \eta v\right|^2}\\
&\leq &\sqrt{\frac{1}{|Q|}\int_Q\left|(1-\eta)v-\frac{1}{|Q|}\int_Q (1-\eta)v\right|^2}+\sqrt{\frac{1}{|Q|}\int_Q\eta^2}.
\eee In the last line we used the fact that $w\mapsto w-\frac{1}{|Q|}\int_Qw$ is a projection operator.
Applying \eqref{L2-eta} and choosing $c_n$ small enough, we can absorb the last term on the left. We can then proceed to estimate\bee
&&\sqrt{\frac{1}{|Q|}\int_Q\left|(1-\eta)v-\frac{1}{|Q|}\int_Q (1-\eta)v\right|^2} \\
&\leq& \sqrt{\frac{1}{|Q|}\int_Q\left|(1-\eta)(v-u)-\frac{1}{|Q|}\int_Q (1-\eta)(v-u)\right|^2}\\
&+&\sqrt{\frac{1}{|Q|}\int_Q\left|(1-\eta)u-\frac{1}{|Q|}\int_Q (1-\eta)u\right|^2}\\ 
&\leq& \sqrt{\frac{1}{|Q|}\int_{Q\setminus F}|v-u|^2}+C_n\ell\sqrt{\frac{1}{|Q|}\int_Q|\nabla[(1-\eta)u]|^2}.
\eee In the last line we used the fact that $\eta=1$ on $F$ and Poincar\'e inequality. To justify its application, one can notice that $u$ is in the Sobolev space $H^1(\mathring{Q}\setminus F,\C^d)$, because $u$ is smooth on $\Omega=\{\psi\neq0\}$ and $\overline{Q\setminus F}\cap\Omega^c=\varnothing$, hence $(1-\eta)u\in H^1(\mathring{Q},\C^d)$.

Using \eqref{def-v} and Leibnitz rule, \bee
\frac{\omega_\infty(\ell)}{4} &\leq& \sqrt{c_n}\cdot\omega_\infty(\ell)+C_n\ell\sqrt{\frac{1}{|Q|}\int_Q|\nabla\eta|^2}+C_n\ell\sqrt{\frac{1}{|Q|}\int_{Q\setminus F}|\nabla u|^2}\\ 
&\leq& C_n\sqrt{c_n}\cdot\omega_\infty(\ell)+C_n\ell\sqrt{\frac{1}{|Q|}\int_{Q\setminus F}|\nabla u|^2},
\eee where in the last inequality we used \eqref{grad-eta}. If $c_n$ is appropriately small, we finally obtain \eqref{grad-u}. The proof is concluded.

\subsection{Examples of potentials satisfying the hypotheses of Theorem \ref{deg-discrete}}\label{example-deg-subsection}
In this subsection we explicitly describe a class of potentials that satisfy the conditions of Theorem \ref{deg-discrete}. 

We begin by describing a partition (up to sets of measure zero) $\mathcal{Q}$ of $\R^n$ into cubes whose sides become smaller and smaller when one moves away from the origin. Fix two sequences of natural numbers $\{a_m\}_{m\geq 1}$ and $\{b_m\}_{m\geq 1}$. Assume that $a_m\geq 2$ and that $\{b_m\}_{m\geq1}$ is increasing. Then $\mathcal{Q}:=\cup_{m\geq 1}\mathcal{Q}_m$, where \be
\mathcal{Q}_m:=\{Q\in\mathcal{G}_{1/(a_1\cdots a_m)}\ \colon\ Q\subseteq \overline{Q(0,2b_m)\setminus Q(0,2b_{m-1})}\},
\ee and $b_0:=0$. We are partitioning the cube $Q(0,2b_1)$ using the cubes of the canonical grid of step $1/a_1$ (and the factor $2$ is needed to ensure that this is possible). Then we are partitioning the region between $Q(0,2b_1)$ and $Q(0,2b_2)$ using the smaller cubes of the canonical grid of step $1/(a_1\cdot a_2)$, and so on. The resulting family $\mathcal{Q}$ satisfies the following property:\begin{enumerate}
\item[($\star$)] if $Q\in \mathcal{G}_{1/(a_1\cdots a_m)}$ and is not contained in $Q(0,2b_{m-1})$, then it is an essentially (i.e., up to sets of measure $0$) disjoint union of cubes of $\mathcal{Q}$,
\end{enumerate} 

Now fix a finite family $\{\mathcal{S}_j\}_{j=1}^N$ of non trivial subspaces of $\C^d$ such that \be\cap_{j=1}^N\mathcal{S}_j=\{0\},\ee and a disjoint partition $[0,1]^n=\cup_{j=1}^NA_j$ into parts of strictly positive measure. Let $\mathcal{S}_0(x):=\mathcal{S}_j$ for every $x\in A_j$. For every $Q\in\mathcal{Q}$, we fix an affine transformation $T_Q$ of $\R^n$ mapping $Q$ onto $[0,1]^n$, and we define $\mathcal{S}_Q(x):=\mathcal{S}_0\circ T_Q(x)$ ($x\in Q$). We can now build the mapping $\mathcal{S}:\R^n\rightarrow \mathcal{V}(d)$ gluing together the various $\mathcal{S}_Q$'s as $Q$ varies in $\mathcal{Q}$. Notice that $\mathcal{S}$ is almost everywhere defined (because the $Q$'s intersect on a set of measure $0$) and it is measurable. 

We can finally define our potentials. Let $\gamma\in L^1_{loc}(\R^n,[0,+\infty))$ be such that $\lim_{x\rightarrow\infty}\gamma(x)=+\infty$. Let $\mathcal{L}(x)$ be the orthogonal in $\C^d$ to $\mathcal{S}(x)$ and define:\be
V(x)v:=0 \quad\forall v\in \mathcal{S}(x),\qquad V(x)v:=\gamma(x)v \quad\forall v\in \mathcal{L}(x).
\ee It is easy to check that $\mathcal{L}$ is measurable and that $V\in L^1_{loc}(\R^n,H_d^{\geq0})$ satisfies condition (i) of Theorem \ref{deg-discrete}. Moreover, by our construction\bel\label{unif-pos}
\inf_{m\in\N}\omega_\infty\left(\frac{1}{a_1\cdots a_m},\mathcal{S}\right)>0, 
\eel which immediately implies condition (ii). To verify this, fix $m$ and notice that, by property ($\star$) of $\mathcal{Q}$, any given cube $Q$ in the canonical grid of step $1/(a_1\cdots a_m)$ which is far enough from the origin is the union of a finite collection $\{Q_k\}$ of cubes in $\mathcal{Q}$. Recalling the definition of $\mathcal{S}$, we can write\be
Q=\cup_{j=1}^N A_j(Q),\ \text{ where } A_j(Q):=\cup_k T_{Q_k}^{-1}(A_j).
\ee
It is immediate to verify that $\mathcal{S}\equiv\mathcal{S}_j$ on $A_j(Q)$, and that $d_j=|A_j(Q)|/|Q|$ is independent of the cube $Q$ and of $m$. By part (ii) of Lemma \ref{angle}, we conclude that $\omega(Q,\mathcal{S})$ is bounded below by a constant that is independent of $Q$ and $m$. Hence \eqref{unif-pos} holds and the hypotheses of Theorem \ref{deg-discrete} are verified by $V$. 

Choosing in particular $\mathcal{S}_j=\text{span}\{e_k:k\neq j\}$ ($j=1\dots,d$), where $\{e_k\}$ is the canonical basis of $\C^d$, one can exhibit a potential $V$ that has everywhere rank $1$ and such that $H_V$ has discrete spectrum, showing that the reverse implication of \eqref{comparison} does not hold in general.

\section{Non-negative polynomial $2\times 2$ potentials}\label{pol-section}

\subsection{Setting} The goal of this section is to prove Theorem \ref{pol2}. Let then $V:\R^n\longrightarrow H_2^{\geq0}$ be a fixed non-negative $2\times 2$ matrix-valued potential such that $(V(x)u,u)$ is a real polynomial for every $u\in\R^2$. We assume that $H_V$ has discrete spectrum, and our task is to show that the scalar operator $H_\lambda$ has discrete spectrum too. 

Recall that $\lambda(x)$ is the minimal eigenvalue of $V(x)$. We call  the maximal eigenvalue $\mu(x)$, the trace $\text{tr}(x)=\lambda(x)+\mu(x)$ and the determinant $\det(x)=\lambda(x)\mu(x)$. Notice that $\text{tr}(x)$ and $\det(x)$ are polynomials, while $\lambda(x)$ and $\mu(x)$ typically are not polynomials.

To prove the discreteness of the spectrum of $H_\lambda$ we resort to the characterization of \cite{ma-sh}.

\begin{thm}\label{maz'ya-shubin} Let $W\in L^1_{\text{loc}}(\R^n)$ be scalar and non-negative. 

Assume that there exists $\gamma>0$ such that\bel\label{toprovepol2}
\lim_{x\rightarrow\infty}\  \inf_{F\in \mathcal{N}_\gamma(Q(x,\ell))}\int_{Q(x,\ell)\setminus F}W(y)dy=+\infty\qquad\forall \ell>0,
\eel where $\mathcal{N}_\gamma(Q)$ is the set of $\gamma$-negligible subsets of $Q$, i.e., compact subsets $F$ of $Q$ such that $\text{Cap}(F)\leq\gamma \text{Cap}(Q)$. Then $H_W$ has discrete spectrum. 

Viceversa, if $H_W$ has discrete spectrum, then \eqref{toprovepol2} holds for any $\gamma<1$.
\end{thm}
Here the capacity of a subset of $Q$ is defined by formulas \eqref{cap3}, \eqref{cap2}, and the comments following them. In particular, if $n=2$, it is the capacity relative to $Q(x,2\ell)$, while 
if $n=1$, for every $\gamma<1$, $\mathcal{N}_\gamma(Q)=\{\varnothing\}$ and the characterization is simply in terms of the integral $\int_{Q(x,\ell)}W(y)dy$.

We can now start the proof of Theorem \ref{pol2}.\newline

\par By Proposition \ref{s-t-c-prop}, for every $\eps>0$ we can fix $R(\eps)<+\infty$ such that
\bel\label{hyppol2}
\int_{|y|\geq R(\eps)}|\psi|^2\leq \eps\cdot\mathcal{E}_V(\psi)\qquad\forall \psi\in E_V.
\eel 
We are going to prove the following. 
\begin{prop}\label{key-pol2}
There exist $\gamma, c_1,c_2>0$ depending only on $V$, such that if $\ell>0$, $\eps\leq c_1\ell^2$, $Q(x,\ell)\subseteq \{|y|\geq R(\eps)\}$ and $F\in \mathcal{N}_\gamma(Q(x,\ell))$, then \be
\frac{1}{\ell^n}\int_{Q(x,\ell)\setminus F}\lambda(y)dy\geq c_2\eps^{-1}.
\ee
\end{prop}
Proposition \ref{key-pol2} entails that $\lambda$ verifies condition \eqref{toprovepol2} of Theorem \ref{maz'ya-shubin}, and hence it allows to conclude that $H_\lambda$ has discrete spectrum. We prove it by testing \eqref{hyppol2} on appropriate test functions. To define them, we exploit a dichotomy to which the next subsection is dedicated.

\subsection{Good and bad cubes}

\begin{lem}\label{good-bad} There is a constant $c>0$ depending only on $n$ and the degree of $\emph{tr}(x)$ and $\det(x)$ such that for every cube $Q$ there is a smaller cube $Q'\subseteq Q$ of side $\ell_{Q'}= c \ell_Q$ such that $\lambda(x)<\mu(x)$ for every $x\in Q'$ and:\begin{enumerate}
\item[(i)] either $\frac{1}{8}\mu(x)\leq\lambda(x)$ for every $x\in Q'$,
\item[(ii)] or $2\lambda(x)\leq \mu(x)$ for every $x\in Q'$ and $\sup_{x\in Q'}\mu(x)\leq 4\inf_{x\in Q'}\mu(x)$.
\end{enumerate}
\end{lem} 

If the first condition above holds we say that $Q$ is a good cube, otherwise we say that $Q$ is a bad cube. The further condition imposed on bad cubes is somehow technical and allows to assume $\mu$ approximately constant.
\par To prove Lemma \ref{good-bad}, we need an elementary lemma expressing the fact that the zero-set of a real-valued polynomial of bounded degree cannot be too dense in a given cube.

\begin{lem}\label{poly-zero}
For every $n,D\in\N$, there exists $c>0$ such that the following holds: if $p$ is a non-zero real-valued polynomial in $n$ variables of degree $\leq D$ and $Q\subseteq \R^n$ is a cube, there is a smaller cube $Q'\subseteq Q$ of side $\ell_{Q'}= c \ell_Q$ that does not intersect the zero set $Z(p)$ of $p$.
\end{lem}

\begin{proof}
We argue by contradiction. Appropriately translating, rotating and rescaling one may find a sequence $\{p_k\}_{k\in\N}$ of non-zero polynomials in $n$ variables of degree $\leq D$ such that $Z(p_k)$ intersects any dyadic cube of side $2^{-k}$ contained in the unit cube $[0,1]^n$. Multiplying by a constant the $p_k$'s, we can also assume that $||p_k||_{L^\infty([0,1]^n)}=1$. 

Since a space of polynomials of bounded degree is finite-dimensional, there is a subsequence $\{p_{k_m}\}_{m\in\N}$ converging uniformly on $[0,1]^n$ to a non-zero polynomial $q$. If $x\in [0,1]^n$, there is a sequence $x_m\in [0,1]^n$ such that $x_m\rightarrow x$ and $p_{k_m}(x_m)=0$. Then \be
|q(x)|\leq |q(x)-q(x_m)|+|q(x_m)-p_{k_m}(x_m)|\leq |q(x)-q(x_m)|+||q-p_{k_m}||_{L^\infty([0,1]^n)}.
\ee Letting $m$ tend to $+\infty$, we get a contradiction.
\end{proof}

\begin{proof}[Proof of Lemma \ref{good-bad}]
Consider the polynomial \be
p(x):=\left(\frac{\text{tr}(x)^2}{8}-\det(x)\right)\cdot\left(\frac{\text{tr}(x)^2}{4}-\det(x)\right).
\ee Given a cube $Q$, Lemma \ref{poly-zero} gives $Q'$ with side $\ell_{Q'}= c\ell_Q$ such that $p(x)\neq0$ on $Q'$. Notice that $c$ depends only on $n$ and the degrees of $\text{tr}(x)$ and $\det(x)$. Since $\det(x)\leq \frac{1}{4}\text{tr}(x)^2$, $\frac{1}{4}\text{tr}(x)^2-\det(x)>0$ on $Q'$, which implies $\lambda(x)<\mu(x)$. For the first term of $p(x)$ there are two options: either $\frac{\text{tr}(x)^2}{8}<\det(x)$ or $\det(x)<\frac{\text{tr}(x)^2}{8}$ ($\forall x\in Q'$). In the first case, \be
\frac{\mu(x)}{8}\leq\frac{\text{tr}(x)}{8}<\frac{\det(x)}{\text{tr}(x)}\leq \lambda(x)\qquad\forall x\in Q'.
\ee In the second case\be
2\lambda(x)\leq 4\frac{\det(x)}{\text{tr}(x)}<\frac{\text{tr}(x)}{2}\leq \mu(x)\qquad\forall x\in Q'.
\ee In the latter case we pass to an even smaller cube where $\mu$ is approximately constant. This can be done observing that $\frac{\text{tr}(x)}{2}\leq \mu(x)\leq\text{tr}(x)$ and using the fact that if $q$ is a non-negative polynomial in $n$ variables of degree $D$ and $Q$ is a cube, there is a smaller cube $Q'\subseteq Q$ of side $c_{n,D}\ell_Q$ such that $\sup_{Q'}q\leq 2\inf_{Q'}q$. This is well-known (see, e.g., page $146$ of \cite{fe-unc}) and can be proved by a compactness-contradiction argument analogous to the proof of Lemma \ref{poly-zero}.
\end{proof}

\subsection{Proof of Proposition \ref{key-pol2}} 

Let $Q(x,\ell)$ be as in Proposition \ref{key-pol2}. Lemma \ref{good-bad} gives us a smaller cube $Q(y,s)$ such that $s= c\ell$, with $c$ depending only on $n$ and the degrees of $\text{tr}(x)$ and $\det(x)$. Modifying by an inessential factor the value of $c$, we can also assume that $Q(y,2s)\subseteq Q(x,\ell)$. For the moment we do not distinguish between good and bad cubes, and we follow the argument in Section 3 of \cite{ma-sh}. Fix $F\in\mathcal{N}_\gamma(Q(x,\ell))$, where $\gamma$ is to be fixed later. Let $F'$ be a compact set such that: \begin{enumerate}
\item[(a)] $F'$ is the closure of an open set with smooth boundary,
\item[(b)] $F\cap Q(y,s)\subseteq F'\subseteq Q(y,\frac{3}{2}s)$,
\item[(c)] $\text{Cap}(F')\leq 2\text{Cap}(F\cap Q(y,s))\leq 2\text{Cap}(F)$. 
\end{enumerate}
In fact, the outer regularity of capacity gives an open set $\Omega\supseteq F\cap Q(y,s)$ such that $\overline{\Omega}$ satisfies conditions $(2)$ and $(3)$, and any $F'$ that is the closure of a smooth open set and lies between $F\cap Q(y,s)$ and $\Omega$ will do. If $n=1$, this discussion becomes trivial and $F'=\varnothing$. If $n\geq 2$, there exists $A<+\infty$ depending only on $c$ (and hence on $V$) such that\bel\label{capF'}
\text{Cap}(F')\leq 2\text{Cap}(F)\leq 2\gamma \text{Cap}(Q(x,\ell))\leq2\gamma A\text{Cap}(Q(y,s)),
\eel by the monotonicity of capacity and its behavior under scaling. We can now fix $\gamma$ in such a way that 
\bel \label{value-gamma} 2\gamma A<1.
\eel 
If $n\geq 2$, we denote by $P$ the equilibrium potential of $F'$, i.e., the function satisfying \begin{enumerate}
\item[(a)] $P\in C(\R^n,[0,1])$,
\item[(b)] $P\equiv1$ on $F'$,
\item[(c)] $\lim_{y\rightarrow\infty}P(y)=0$ (if $n\geq3$) or $P(y)\equiv0$ for every $y\notin Q(x,2\ell)$ (if $n=2$),
\item[(d)] $\Delta P=0$ on $\R^n\setminus F'$ (if $n\geq3$) or on $Q(x,2\ell)\setminus F'$ (if $n=2$),
\item[(e)] $\int_{\R^n}|\nabla P|^2=\text{Cap}(F')$.
\end{enumerate} The existence of $P$ is a classical fact of potential theory (see \cite{wermer}). If $n=1$, put $P\equiv0$ in what follows. We now state as a lemma a useful property of $P$.

\begin{lem}\label{bound-eqpot}
$\int_{Q(y,s)}(1-P)^2\geq c_n(c\ell)^n.$
\end{lem}
\begin{proof} If $n\geq 3$, Lemma $3.1$ of \cite{ko-ma-sh}, or inequality $3.11$ of \cite{ma-sh}, asserts that if $Q$ is a cube, $E\subseteq \frac{3}{2}Q$ is the closure of a smooth open set such that $\text{Cap}(E)<\text{Cap}(Q)$ and $P_E$ is the equilibrium potential of $E$, then
\bel\label{eqpot}
c_n\left(1-\frac{\text{Cap}(E)}{\text{Cap}(Q)}\right)^2\leq \sigma\frac{\text{Cap}(E)}{\text{Cap}(Q)}+\frac{\sigma^{-1}}{|Q|}\int_Q(1-P_E)^2\qquad\forall \sigma>0.
\eel The thesis follows applying this to $E=F'$ and $Q=Q(y,s)$, recalling \eqref{value-gamma} and choosing $\sigma$ small enough (depending on $n$).

The case $n=1$ is trivial, while the case $n=2$ can be derived using the argument above with a minor observation: in \eqref{eqpot} $\text{Cap}(E)$ is the capacity relative to $2Q$, while here we are considering $F'=E$ and we want $\text{Cap}(F')$ to denote capacity relative to $Q(x,2\ell)$. This is not a problem, because the two quantities are easily seen to be comparable up to a constant depending on $c$ and hence, taking $\gamma$ small and recalling \eqref{capF'} and \eqref{value-gamma}, we conclude.
\end{proof}

Let $\eta_\delta\in C^\infty_c(\R^n,[0,1])$ be a scalar test function equal to $1$ on $Q(y,(1-\delta) s)$, equal to $0$ outside of $Q(y,s)$, and such that $|\nabla \eta_\delta|\leq C_n(\delta s)^{-1}=C_n(\delta c\ell)^{-1} $. 

We now analyze separately good and bad cubes.\newline

If $Q(x,\ell)$ is good, we define $\psi:=\eta_\delta(1-P)e_1$, where $e_1=(1,0)$ (but any other unit norm vector of $\C^2$ would do). Notice that $\psi\in E_V$. An integration by parts and the harmonicity of $P$ outside $F'$ yield\bee
\int_{\R^n}|\nabla\psi|^2&=&\int_{\R^n}|\nabla\eta_\delta|^2(1-P)^2+\int_{\R^n}\eta_\delta^2|\nabla P|^2-\int_{\R^n}\nabla (\eta_\delta^2)\cdot (1-P)\nabla P\\
&=&\int_{\R^n}|\nabla\eta_\delta|^2(1-P)^2\leq C_n(\delta c\ell)^{-2}\int_{Q(y,s)}(1-P)^2\leq C_n \delta^{-2} (c\ell)^{n-2},
\eee where the first inequality follows from the properties of $\eta_\delta$, while the second follows from the fact that $0\leq P\leq1$.
Since $Q(x,\ell)\subseteq\{|y|\geq R(\eps)\}$, we test \eqref{hyppol2} on $\psi$, using the bound above. The result is:
\bee
\int_{\R^n} \eta_\delta^2(1-P)^2 &\leq& \eps\left(C_n\delta^{-2} (c\ell)^{n-2}+\int_{\R^n}\eta_\delta^2(1-P)^2(Ve_1,e_1)\right)\\
&\leq& \eps\left(C_n\delta^{-2} (c\ell)^{n-2}+8\int_{Q(y,s)\setminus F'}\lambda\right),
\eee where in the second line we used the fact that $P\equiv1$ on $F'$ and the estimate $(Ve_1,e_1)\leq \mu\leq 8\lambda$, which holds on $Q(y,s)$ because $Q(x,\ell)$ is good. Since $F'\supseteq F\cap Q(y,s)$, we have\bel\label{good-1}
\int_{\R^n} \eta_\delta^2(1-P)^2 \leq\eps\left(C_n\delta^{-2} (c\ell)^{n-2}+8\int_{Q(x,\ell)\setminus F}\lambda\right)
\eel
Observing that $|Q(y,s)\setminus Q(y,(1-\delta)s)|\leq C_n\delta s^n=C_n\delta (c\ell)^n$ and $\eta_\delta\equiv1$ on $Q(y,(1-\delta)s)$ we can bound\be
\int_{Q(y,s)}(1-P)^2\leq\int_{\R^n}\eta_\delta^2(1-P)^2+C_n\delta (c\ell)^n.
\ee
By Lemma \ref{bound-eqpot} we can choose $\delta=\delta_n$ small enough and obtain \bel\label{good-2}c_n(c\ell)^n\leq \int_{\R^n}\eta_{\delta_n}^2(1-P)^2.\eel Putting \eqref{good-1} (for $\delta=\delta_n$) and \eqref{good-2} together, we find\be
c_n(c\ell)^n\leq \eps\left(C_n(c\ell)^{n-2}+8\int_{Q(x,\ell)\setminus F}\lambda\right).
\ee It is now clear that there exist $c_1,c_2>0$ depending only on $V$ such that if $\eps\leq c_1\ell^2$, then $\frac{1}{\ell^n}\int_{Q(x,\ell)\setminus F}\lambda\geq c_2\eps^{-1}$, as we wanted.

Notice that this argument is a modification of Maz'ya-Shubin necessity argument (see Section $3.$ of \cite{ma-sh}). We now move to the bad cubes, which require more work.\newline

If $Q(x,\ell)$ is bad the main difficulty is that $(Ve_1,e_1)$ is not bounded by $\lambda$ on $Q(y,s)$. The natural idea is to test \eqref{hyppol2} on $\psi:=\eta_\delta(1-P)v$, where $v:Q(y,s)\rightarrow \R^2$ is smooth and satisfies $Vv\equiv\lambda v$ and $|v|=1$ everywhere. It is possible to define such a smooth selection of eigenvectors because the cube is simply-connected and $\lambda<\mu$ on $Q(y,s)$, i.e., there are no eigenvalue crossings (by Lemma \ref{good-bad}). Notice that $\lambda$ is smooth on $Q(y,s)$ for the same reason.

Observe that\bee
\int_{\R^n}|\nabla\psi|^2&\leq&2\int_{\R^n}|\nabla(\eta_\delta(1-P))|^2+2\int_{\R^n}\eta_\delta^2(1-P)^2|\nabla v|^2\\
&\leq& C_n \delta^{-2} (c\ell)^{n-2}+2\int_{Q(y,s)}|\nabla v|^2,
\eee where the last bound is proved exactly as for good cubes. We claim that \bel\label{grad-v}
\int_{Q(y,s)}|\nabla v|^2\leq C\ell^{n-2},
\eel where $C$ is a constant depending on $V$, but independent of all the other parameters.
The reader may easily check that the argument given for good cubes may then be carried out word by word to prove Proposition \ref{key-pol2} also for bad cubes. To prove the claim \eqref{grad-v}, we need an elementary inequality.

\begin{lem}\label{grad-poly-pot}
There is a constant $C_1$ depending only on $V$ such that for every cube $Q$\be \int_Q||\partial_jV||_{op}^2\leq C_1\ell_Q^{-2}\int_Q||V||_{\text{op}}^2\qquad\forall j=1,\dots,n.\ee
\end{lem}

\begin{proof}
The set $X_{n,d,D}$ of functions $W:\R^n\rightarrow H_d$ whose matrix elements are polynomials of degree $\leq D$ is a finite-dimensional vector space and $\partial_j$ is a linear operator on $X_{n,d,D}$. Since $\sqrt{\int_{Q(0,1)}||W||_{op}^2}$ is a norm on $X_{n,d,D}$, we have the bound\be
\int_{Q(0,1)}||\partial_jW||_{op}^2\leq C_1\int_{Q(0,1)}||W||_{op}^2 \qquad\forall W\in X_{n,d,D}.
\ee The application of this inequality to $V$, after rescaling, rotating and translating, gives the thesis.
\end{proof}

Let us differentiate the identity $Vv=\lambda v$. Notice that $V$, $v$ and $\lambda$ are real-valued and that all the computations are on $Q(y,s)$. We obtain\bel\label{diff-eig}
(\partial_jV)v+V(\partial_jv)=(\partial_j\lambda) v+\lambda(\partial_j v).
\eel Since $|v|=1$, $\partial_jv$ is pointwise orthogonal to $v$. Taking the real scalar product of both sides of \eqref{diff-eig} with $\partial_jv$, we obtain\be
((\partial_jV)v,\partial_jv)+(V(\partial_jv),\partial_jv)=\lambda|\partial_j v|^2.
\ee Since $\partial_jv$ and $v$ are orthogonal vectors in $\R^2$, $\partial_jv$ must be an eigenvector of eigenvalue $\mu$ at every point of $Q(y,s)$. Hence $(V(\partial_jv),\partial_jv)=\mu|\partial_jv|^2$ and\be
(\mu-\lambda)|\partial_jv|^2=-(\partial_jVv,\partial_jv)\leq ||\partial_jV||_{op}|\partial_jv|.
\ee Since $\mu\geq2\lambda$ and $4\mu\geq ||\mu||_{L^\infty(Q(y,s))}$ on $Q(y,s)$, this implies \be||\mu||_{L^\infty(Q(y,s))}|\partial_jv|\leq 8||\partial_jV||_{op}.\ee Squaring and integrating on $Q(y,s)$ this pointwise inequality, and applying Lemma \ref{grad-poly-pot}, we find\bee
&&||\mu||_{L^\infty(Q(y,s))}^2\int_{Q(y,s)}|\partial_jv|^2\leq64 \int_{Q(y,s)}||\partial_jV||_{op}^2\\ 
&\leq& 64C_1(c\ell)^{-2}\int_{Q(y,s)}||V||_{op}^2\leq 64C_1(c\ell)^{n-2}||\mu||_{L^\infty(Q(y,s))}^2,
\eee where we used the fact that $||V||_{op}=\mu$. Simplifying and summing over $j$, we get \eqref{grad-v} and hence Theorem \ref{pol2}.

\section{A natural extension of the Maz'ya-Shubin condition\\ is not sufficient when $d\geq2$}\label{final}

Our discussion until this point leaves open the possibility that a natural extension of Maz'ya-Shubin's result may hold for matrix-valued potentials. In particular, the statements of Theorem \ref{Ainfty-thm} and Theorem \ref{maz'ya-shubin} may suggest the conjecture that a necessary and sufficient condition for discreteness of the spectrum of $H_V$ may be the validity of\bel\label{wrong-condition}
\lim_{x\rightarrow\infty}\ \inf_{F\in\mathcal{N}_\gamma(Q(x,\ell))}\lambda\left(\int_{Q(x,\ell)\setminus F} V(y)dy\right)=+\infty\qquad\forall \ell>0,
\eel 
for some, and hence every, $\gamma\in(0,1)$.

This condition is indeed necessary. This can be seen adapting the necessity argument of \cite{ma-sh}, in the same spirit of our proof of the implication (i)$\Rightarrow$(ii) of Theorem \ref{Ainfty-thm}. More precisely, one can take the argument of page $931$ of \cite{ma-sh} and replace the scalar test function $\chi_\delta (1-P_F)$ with the vector-valued test function $\chi_\delta (1-P_F)v$, where $v\in\C^d$ has norm $1$ and is otherwise arbitrary. Carrying out their computations and minimizing in $v$ at the end as in the aforementioned proof of Theorem \ref{Ainfty-thm}, one sees that if the vector-valued operator $H_V$ has discrete spectrum, then $V$ satisfies condition \eqref{wrong-condition}.

Nevertheless, the converse does not hold. In Section \ref{Ainfty-section} we have seen that for every $n\geq1$ and $d\geq2$, there is a non-negative $d\times d$ polynomial potential $W_{n,d}$ on $\R^n$ such that $H_{W_{n,d}}$ has not discrete spectrum, $W_{n,d}$ satisfies condition \eqref{condition-b}, and $\lim_{x\rightarrow\infty}\lambda\left(\int_{Q(x,\ell)}W_{n,d}\right)=+\infty$ for every $\ell>0$. We are going to show that such a potential $W_{n,d}$ satisfies \eqref{wrong-condition} for $\gamma$ small enough.

Assume that $n\geq3$. If $F\in\mathcal{N}_\gamma(Q(x,\ell))$, we can use the comparison between Lebesgue measure and capacity
\be |F|\leq C_n\text{Cap}(F)^\frac{n}{n-2},\ee
and the fact that $\text{Cap}(Q(x,\ell))=\ell^{n-2}\text{Cap}(Q(0,1))$, to conclude that \be|F|\leq C_n\gamma^\frac{n}{n-2}|Q(x,\ell)|.\ee If $\gamma$ is small, condition \eqref{condition-b} implies that $\int_{Q(x,\ell)\setminus F}W_{n,d}\geq \beta\int_{Q(x,\ell)}W_{n,d}$ for some $\beta>0$ depending on $W_{n,d}$, and hence \be
\lim_{x\rightarrow\infty}\ \inf_{F\in\mathcal{N}_\gamma(Q(x,\ell))}\lambda\left(\int_{Q(x,\ell)\setminus F}U\right)\geq \beta\cdot\lim_{x\rightarrow\infty}\lambda\left(\int_{Q(x,\ell)}U\right)=+\infty,\ee as we wanted. We omit the elementary observations needed to cover also the cases $n=1$ and $n=2$.

\section*{Acknowledgements}

I would like to thank Fulvio Ricci for his support and many helpful discussions on the subject of this work.

\bibliographystyle{amsalpha}
\bibliography{schrod-biblio}

\end{document}